\documentclass[a4paper,american,reqno]{amsart}

\usepackage{babel}
\usepackage[utf8]{inputenc}
\usepackage[T1]{fontenc}
\usepackage{lmodern}
\usepackage{todonotes}
\usepackage{microtype}
\usepackage{enumitem}
\usepackage{booktabs}
\usepackage{pgfplots}
\usepackage{relsize}
\usepackage{tikz}
\usepackage{xspace}
\usepackage{csquotes}
\usepackage{hyperref} 
\usepackage{algorithm}
\usepackage{algorithmic} 
\usepackage{cite}

\newtheorem{theorem}{Theorem}
\newtheorem{corollary}{Corollary}
\newtheorem{assumption}{Assumption}
\theoremstyle{definition}

\newtheorem{example}{Example}

\makeatletter
\patchcmd{\@settitle}{\uppercasenonmath\@title}{\scshape\large}{}{}
\patchcmd{\@setauthors}{\MakeUppercase}{\scshape\normalsize}{}{}
\makeatother

\vfuzz \hfuzz
\newcommand{\ie}{i.e.,\xspace}  

\newcommand{\eg}{e.g.,\xspace}  

\newcommand{\st}{\text{s.t.}\xspace}
\newcommand{\abs}[1]{\mathchoice{\left\lvert #1 \right\rvert}%
{\lvert#1\rvert}%
{\lvert#1\rvert}%
{\lvert#1\rvert}%
}

\renewcommand{\:}{\colon}
\newcommand{\Yt}{\ensuremath{Y_t}}
\newcommand{\Ut}{\ensuremath{U_t}}
\newcommand{\Vt}{\ensuremath{V_t}}
\newcommand{\Zt}{\ensuremath{Z_t}}
\newcommand{\Wt}{\ensuremath{W_t}}

\newcommand{\RR}{\ensuremath{\mathbb{R}}}
\newcommand{\NN}{\ensuremath{\mathbb{N}}}

\DeclareMathOperator{\Vp}{\raisebox{-0.5em}{\ensuremath{+}}\kern-1.0ex\bigvee}
\DeclareMathOperator{\Vm}{\raisebox{-0.5em}{\ensuremath{--}}\kern-1.0ex\bigvee}

\newcommand{\upto}{\nearrow}
\newcommand{\lOneNorm}[1]{\ensuremath{\mathchoice{\left\lvert #1 \right\rvert}{\lvert#1\rvert}{\lvert#1\rVert}{\lvert#1\rvert}_{l_1}}}

\usepgfplotslibrary{groupplots}
\usetikzlibrary{external}
\usetikzlibrary{arrows}
\usetikzlibrary{decorations.pathmorphing}
{pde-relax-adm_pre} 

\tikzexternaldisable

\begin{document}

\title[Penalty ADM for mixed-integer optimal control]{Penalty alternating direction methods for mixed-integer optimal control with combinatorial constraints}
\author[S. Göttlich, F.M. Hante, A. Potschka, L. Schewe]{Simone~Göttlich$^{1}$, Falk~M.~Hante$^{2}$, Andreas~Potschka$^{3}$, Lars~Schewe$^{4}$}
\address{
  $^1$ Universität Mannheim; 
  $^2$ Humboldt-Universität zu Berlin, Unter den Linden 6, 10099 Berlin, Germany;
  $^3$ Interdisciplinary Center for Scientific Computing, Heidelberg University,
  Im Neuenheimer Feld~205, 69120~Heidelberg, Germany; 
  $^4$    The University of Edinburgh, School of Mathematics, James Clerk Maxwell
   Building, Peter Guthrie Tait
   Road, Edinburgh, EH9 3FD, UK}

\date{\today}

\begin{abstract}
We consider mixed-integer optimal control problems with combinatorial constraints that couple over time such as minimum dwell times. We analyze a lifting and decomposition approach into a mixed-integer optimal control problem without combinatorial constraints and a mixed-integer problem for the combinatorial constraints in the control space. Both problems can be solved very efficiently with existing methods such as outer convexification with sum-up-rounding strategies and mixed-integer linear programming techniques. The coupling is handled using a penalty-approach. We provide an exactness result for the penalty which yields a solution approach that convergences to partial minima. We compare the quality of these dedicated points with those of other heuristics amongst an academic example and also for the optimization of electric transmission lines with switching of the network topology for flow reallocation in order to satisfy demands.

\end{abstract}

\subjclass[2010]{49J15, 49J20, 65K05, 90C11
}

\maketitle

\smallskip
\noindent \textbf{Keywords.}  mixed-integer optimization, partial differential equations, dwell-time constraints, alternating direction methods, penalty methods

\section{Introduction}\label{sec:intro}
Optimal control problems subject to integer restrictions on some part of the controls have recently received 
a lot of attention in the literature. This problem class is a convenient way to model, for instance, autonomous 
driving in case of vehicles with gear shift power units \cite{KirchesEtAl2010}, contact problems 
such as robotic multi-arm transport \cite{BussEtAl2002}, or the operation of networked infrastructure systems
such as gas pipelines \cite{Hante2019}, water canals \cite{HanteEtAl2017}, traffic roads \cite{GoettlichPotschkaZiegler2017}, and power transmission lines
\cite{GoettlichPotschkaTeuber2019} with switching of valves, gates, traffic lights and interconnectors, respectively.   
Often, the integer controls are additionally constrained in order to prevent certain switching configurations, 
to limit the number of switches or to enforce certain dwell or dead times after a switch. In this paper, we 
consider such combinatorial constraints with a focus on conditions which cannot be imposed pointwise and hence 
couple over time. 

A discretization of such problems naturally leads to mixed-integer non-linear programs that often become
computationally intractable when the discretization stepsizes tend to zero. Moreover, when passing to the limit, 
one may face convergence issues \cite{HanteSchmidt2019}. A computationally much more efficient alternative
solution approach is based on decomposition techniques, splitting the problem into a continuous subproblem by partial
outer convexification with relaxation of binary multipliers (POC) combined with a combinatorial integral approximation 
problem (CIAP) \cite{Sager2006,Kirches2010,SagerJungKirches2011,SagerBockDiehl2012,HanteSager2013,JungReineltSager2015,MannsKirches2019}.
However, in the presence of combinatorial constraints that couple over time, this approach only yields feasible
solutions with a priori lower bounds, but yet without any characterization of optimality in some reasonable sense
on the mixed-integer level.

We consider here another approach based on the idea of alternating direction methods (ADM). This will provide feasible 
solutions which can be characterized as partially optimal in a lifted sense. The approach uses POC and CIAP in one direction and a 
mixed-integer linear problem (differing from the combinatorial integral approximation problem) in another direction.
Both directions are weakly coupled using a penalty term with adapting an idea outlined in \cite{geissler.ea:fpump-adm}. 
Based on exactness of this penalty, we provide a convergence
result of this method. Our analysis applies to problems in the setting of abstract semilinear evolutions subject to 
control constraints. However, the methods can be extended to state constraints. In particular, the techniques apply 
to optimal control problems with ordinary differential equations.
The method can be also seen as an adaptation of classical feasibility pump
algorithms (for an overview, see~\cite{Berthold2019}) with heavy structure exploitation for mixed-integer optimal control problems.

One feature of our approach is a clear separation of the combinatorial aspects from the continuous control aspects of the problem.
This is in contrast to, e.g., the approach proposed in \cite{Takapoui20165619,Takapoui20202}. There, the full problem is discretized
and then a variant of an Alternating Direction Method of Multipliers (ADMM) is used to obtain heuristic solutions. Further recent applications
of ADM type methods are related to electricity networks, see e.g.~\cite{Braun2018,Engelmann2020,Magnusson2015}.

In a numerical study, we consider two problems from the \href{https://mintoc.de}{mintOC.de} library \cite{sager2012benchmark}, 
which we augment by minimum dwell-time constraints. We compare the proposed approach with direct discretization and mixed-integer 
programming techniques in order to address local vs. global optimality and to the decomposition with POC and CIAP as a heuristic.

We note that continuous reformulations of such problems with switching time and mode insertion optimization or combinatorial constraints 
can also be seen as an alternating direction method \cite{DeMarchi2019,Palagachev2017,RuefflerHante2016,RuefflerHanteMehrmann2018}, but
that the approach proposed here is different.

The article is organized as follows. In Section~\ref{sec:problem}, we present the problem formulation. In Section~\ref{sec:epsADM},
we extend the framework of alternating direction methods to partial $\varepsilon$-optimality. In Section~\ref{sec:ADMforMIOCP},
we apply the $\varepsilon$-ADM framework together with penalty techniques to mixed-integer optimal control problems as the main algorithm and 
develop the convergence theory. In Section~\ref{sec:numstudy}, we provide our numerical results. In Section~\ref{sec:conclusion}, 
we give concluding remarks.


\section{Problem formulation} \label{sec:problem}
We consider a mixed-integer optimal control problem of the form
\begin{subequations}
  \label{eq:abstract-problem}
\begin{alignat}{2}
 \min_{y,u,v}\quad &\Phi(y(T))\label{eq:cost}\\
 \st~&\dot{y}(t)=Ay(t)+f(y(t),u(t),v(t))\quad &&\text{on}~Y,~t \in (0,T),~y(0)=y_0,\label{eq:StateEq}\\
 &u \in \Sigma \subset \Ut, \label{eq:CCC}\\
 &v \in \Gamma \subset \Vt,\label{eq:DCC}\\
 &y \in \Yt,\label{eq:spaces}
\end{alignat}
\end{subequations}
where for some $M\in\NN$ and $p \in (0,\infty]$, $Y$ and $U$ are Banach spaces, $\Yt=C([0,T];Y)$, $\Ut=L^p(0,T;U)$, 
$\Vt=L^\infty(0,T;\{0,1\}^M)$,
$A$ is a (densely defined) linear operator on $Y$, $f$ is a nonlinear mapping $f\: Y \times U \times \{0,1\}^M \to Y$,
$\Phi$ is a nonlinear function $\Phi\: Y \to \RR \cup\{\infty\}$ representing state costs, 
$\Sigma$ is a subset of $\Ut$ representing constraints on the continuous control $u$,
and $\Gamma$ is a subset of $\Vt$ representing combinatorial constraints 
(e.\,g., dwell-time constraints and switching order constraints).

We say that the set of combinatorial constraints $\Gamma$ has a \emph{uniform finiteness property}, if there exists a constant $n_s \in \NN$ 
such that $v \in \Gamma$ implies that $v$ is piecewise constant with at most $n_s$ switching points.

\begin{example}[Combinatorial constraints]
With $v=(v_1,\ldots,v_M)$ being the componentwise respresentation of $v \in \Vt$ and the total variation of 
the $i$th component on the interval $(t_1,t_2) \subset (0,T)$ being denoted by
\begin{equation*}
 |v_i|_{(t_1,t_2)} = \sup\left\{ \int_{t_1}^{t_2} v_i(t) \phi'(t)\,dt : \phi \in C^1_c(t_1,t_2), \|\phi\|_\infty \leq 1 \right\},
\end{equation*}
where $C^1_c(t_1,t_2)$ denotes continuously differentiable vector functions of compact support on $(t_1,t_2)$, 
we can for example enforce a minimal dwell-time $\tau_{\min}$ with the constraint
\begin{equation}\label{eq:mindwelltimeconstraint}
v_i(t+\tau_{\min})-v_i(t)+|v_i|_{(t,t+\tau_{\min})} \leq 2 ~\text{for all}~t \in (0,T-\tau_{\min}),~i=1,\ldots,M
\end{equation}
or directly limit the total number of switches for the $i$th component to $n_s^{\max}$ by
\begin{equation}\label{eq:numswitchconstraint}
 |v_i|_{(0,T)} \leq n_s^{\max},~i=1,\ldots,M.
\end{equation}
In both cases it is easy to see that any set $\Gamma \subset \Vt$ containing one of these constraints has the uniform finiteness property.
Further additional constraints are of course possible, for instance, a maximum dwell-time $\tau_{\max}$ for a subset of 
components $I \subset \{1,\ldots,M\}$
\begin{equation}\label{eq:maxdwelltimeconstraint}
  v_i(t)-v_i(t-\tau_{\max})+|v_i|_{(t-\tau_{\max},t)} \geq 2 ~\text{for all}~t \in (\tau_{\max},T),~i \in I.
\end{equation}
Constraints of the form \eqref{eq:mindwelltimeconstraint}--\eqref{eq:maxdwelltimeconstraint} or variants of it are significant 
in many applications that involve switching control, but they are typically extremely difficult to be treated in the context of mixed-integer 
optimal control because they are not defined pointwise in time.
\end{example}

Concerning the wellposedness of the problem, we make the following assumptions.

\begin{assumption}\label{ass:main} Suppose that $A$ generates a strongly continuous semigroup $e^{tA}$ on $Y$, and that there exists a constant $K>0$ such that,
for all $v \in \{0,1\}^M$,
\begin{itemize}
 \item[i)] the map $(y,u) \mapsto f(y,u,v)$ is continuous on $Y \times U$,
 \item[ii)] $\|f(y,u,v)\|_Y \leq K(1+\|y\|)$ for all $y \in Y$, $u \in U$,
 \item[iii)] $\|f(y,u,v)-f(z,u,v)\|_Y \leq K\|y-z\|_Y$ for all $y,z \in Y$, $u \in U$.
\end{itemize}
Moreover, assume that $\Phi\colon Y \to \mathbb{R}$ is Lipschitz continuous on bounded subsets of $Y$.
\end{assumption}

We note that the conditions of Assumption~\ref{ass:main} are sufficient for the state equation~\eqref{eq:StateEq} to admit a unique solution 
$y(\cdot;u,v)$ in $C([0,T];Y)$ given by
\begin{equation}\label{eq:soldef}
y(t)=e^{tA}y_0 + \int_0^t e^{(t-\tau)A}f(y(\tau),u(\tau),v(\tau))\,d\tau.  
\end{equation}
Of course, other conditions are also possible, see \eg \cite{Pazy}. Moreover, the uniform finiteness property is crucial for the existence of 
optimal solutions for the problem \eqref{eq:abstract-problem}. For additional problem specific assumptions, appropriate wellposedness results can 
for example be obtained via parametric programming. The next theorem illustrates this for the case of generators of immediately compact semigroups.

\begin{theorem}
Suppose that Assumption~\ref{ass:main} holds. Moreover, assume that $X$ is separable and reflexive, $\Sigma=U_t$ and that $f(y,U,v_i)$ is closed and convex in $Y$, $e^{tA}$ is compact for $t>0$ and 
that $\Gamma$ has the uniform finiteness property. Then the problem \eqref{eq:abstract-problem} has an optimal solution $(y^*,u^*,v^*)$.
\end{theorem}
\begin{proof}
Under the uniform finiteness property, the problem \eqref{eq:abstract-problem} can be considered as a parametric two stage problem, where the inner problem
consists of minimizing with respect to $u$ and the outer problem is a minimization with respect to finitely many switching times $\tau_k$. Under the given
assumptions, the inner problem has an optimal solution and the optimal value depends continuously on the initial data \cite{CannarsaFrankowska1992}, and hence 
via \eqref{eq:soldef} on the switching times $\tau_k \in [0,T]$. The claim then follows from the extreme value theorem of Weierstrass.
\end{proof}

For the solution approach considered below, we note that under the Assumption~\ref{ass:main}, we can consider the reduced problem
\begin{subequations}
  \label{eq:abstract-problem-reduced}
\begin{alignat}{2}
 \min_{u,v}\quad &\Psi(u,v):=\Phi(y(T;u,v))\label{eq:cost-red}\\
 \st~& u \in \Sigma \subset \Ut, v \in \Gamma \subset \Vt
\end{alignat}
\end{subequations}
and results for \eqref{eq:abstract-problem-reduced} can be carried over to the original problem \eqref{eq:abstract-problem} via \eqref{eq:soldef}.

\section{ADM with $\varepsilon$-optimality}\label{sec:epsADM}

As a solution approach we extend here the idea of ADM. Suppose we were to minimize a nonlinear function
$\Psi(u,v)$ over $u \in U$ and $v \in V$ subject to constraints $(u,v) \in \Omega$ for some given feasible set $\Omega$. Further suppose that we can compute $\varepsilon$-optimal
solutions for each of the partials $u$ (with $v$ fixed) and $v$ (with $u$ fixed). Then, given some $\varepsilon \geq 0$ and some 
guess $(u^0,v^0)$, we can consider the following sequential approach to compute a solution candidate $(u^*,v^*)$:
\begin{itemize}
 \item[i)] Find $u^{l+1}$ such that $\Psi(u^{l+1},v^l) \leq \Psi(u,v^l) + \frac{\varepsilon}{2}$ for all $(u,v^l) \in \Omega$.
 \item[ii)] If $\Psi(u^{l+1},v^l) \geq \Psi(u^l,v^l)-\frac{\varepsilon}{2}$, set $(u^*,v^*)=(u^l,v^l)$.
 \item[iii)] Find $v^{l+1}$ such that $\Psi(u^{l+1},v^{l+1}) \leq \Psi(u^{l+1},v) + \frac{\varepsilon}{2}$ for all $(u^{l+1},v) \in \Omega$.
 \item[iv)] If $\Psi(u^{l+1},v^{l+1}) \geq \Psi(u^{l+1},v^l)-\frac{\varepsilon}{2}$, set $(u^*,v^*)=(u^{l+1},v^l)$.
 \item[v)] Set $l=l+1$ and continue with step i).
 \end{itemize}
 This algorithm may not terminate.
 For classical ADM, there are well-known conditions under which we can ensure that
 the algorithm does not cycle, i.e.\ that the algorithm does not get stuck in a loop of different solutions;
 for a discussion, see~\cite{geissler.ea:fpump-adm}.
 However, if it terminates, we can conclude that $(u^*,v^*)$ satisfies
\begin{subequations}
 \begin{alignat}{1}
  \Psi(u^*,v^*) &\leq \Psi(u,v^*)+\varepsilon,\quad \text{for all}~(u,v^*) \in \Omega \label{eq:O1}\\
  \Psi(u^*,v^*) &\leq \Psi(u^*,v)+\varepsilon,\quad \text{for all}~(u^*,v) \in \Omega \label{eq:O2}.  
 \end{alignat}
\end{subequations}
This can be seen as follows: If the algorithms terminates in step ii), we have for some $\hat{l}$ from ii)
\begin{equation*}
 \Psi(u^*,v^*)= \Psi(u^{\hat{l}},v^{\hat{l}}) \leq \Psi(u^{\hat{l}+1},v^{\hat{l}}) + \frac{\varepsilon}{2} 
\end{equation*} 
and from i)
\begin{equation*}
\Psi(u^{\hat{l}+1},v^{\hat{l}}) \leq \Psi(u,v^{\hat{l}}) + \frac{\varepsilon}{2},\quad \text{for all}~(u,v^{\hat{l}}) \in \Omega,
\end{equation*}
hence, with $v^{\hat{l}}=v^*$, we get 
\begin{equation*}
  \Psi(u^*,v^*) \leq \Psi(u,v^*) + \varepsilon,\quad \text{for all}~(u,v^*) \in \Omega.
\end{equation*}
Moreover, from step iii) with $l=\hat{l}-1$, we have
\begin{equation*}
 \Psi(u^*,v^*)=\Psi(u^{\hat{l}},v^{\hat{l}}) \leq \Psi(u^{\hat{l}},v)+\frac{\varepsilon}{2},\quad\text{for all}~(u^{\hat{l}},v) \in \Omega,
\end{equation*}
hence, again with $u^{\hat{l}}=u^*$, we have
\begin{equation*}
  \Psi(u^*,v^*) \leq \Psi(u^*,v)+\varepsilon,\quad\text{for all}~(u^*,v)\in\Omega. 
\end{equation*}
If the algorithm terminates in step iv), we have for some $\hat{l}$ from iv)
\begin{equation*}
 \Psi(u^*,v^*)=\Psi(u^{\hat{l}+1},v^{\hat{l}})\leq \Psi(u^{\hat{l}+1},v^{\hat{l}+1})+\frac{\varepsilon}{2}
\end{equation*}
and from iii)
\begin{equation*}
 \Psi(u^{\hat{l}+1},v^{\hat{l}+1}) \leq \Psi(u^{\hat{l}+1},v)+\frac{\varepsilon}{2},\quad\text{for all}~(u^{\hat{l}},v) \in \Omega.
\end{equation*}
Hence, with $u^{\hat{l}+1}=u^*$, we get
\begin{equation*}
  \Psi(u^*,v^*) \leq  \Psi(u^*,v)+\varepsilon,\quad\text{for all}~(u^*,v) \in \Omega.
\end{equation*}
Moreover, from i), we have
\begin{equation*}
 \Psi(u^*,v^*)=\Psi(u^{\hat{l}+1},v^{\hat{l}}) \leq \Psi(u,v^{\hat{l}})+\frac{\varepsilon}{2},\quad\text{for all}~(u,v^{\hat{l}}) \in \Omega.
\end{equation*}
Hence with $v^{\hat{l}}=v^*$, we get
\begin{equation*}
 \Psi(u^*,v^*) \leq \Psi(u,v^*) + \varepsilon,\quad\text{for all}~(u,v^*) \in \Omega.
\end{equation*}
We shall call points $(u^*,v^*)$ satisfying \eqref{eq:O1} and \eqref{eq:O2} p-$\varepsilon$-optimal as a shorthand for partially $\varepsilon$-optimal.


\section{ADM and p-minima for mixed-integer optimal control problems}\label{sec:ADMforMIOCP}
Concerning the mixed-integer optimal control problem \eqref{eq:abstract-problem} or equivalently for the reduced form 
\eqref{eq:abstract-problem-reduced}, a natural ADM splitting is using the directions $u \in U$ 
and $v \in V$. However, in the direction of $v$, this still results in a mixed-integer nonlinear optimization problem subject to a 
differential equation. To avoid this, we will instead use that \eqref{eq:abstract-problem} is equivalent to
\begin{subequations}\label{eq:abstract-problem-split-v-red}
\begin{alignat}{1}
 \min_{u,v,\tilde v}\quad &\Psi(u,v)\label{eq:cost-split-v-red}\\
 \st~&u \in \Sigma \subset \Ut, \label{eq:CCC-split-v}\\
 &v = \tilde v \label{eq:vsplit-split-v}\\
 &\tilde v \in \Gamma \subset \Vt, v \in \Vt,\label{eq:DCC-split-v}
\end{alignat}
\end{subequations}
and consider a splitting with respect to the directions $(u,v)$ and $\tilde{v}$. This particular splitting is chosen deliberately in view of the 
fact that the two subproblems can be efficiently solved to $\varepsilon$-optimallity with existing techniques.
Motivated by \eqref{eq:O1} and \eqref{eq:O2}, we say that a point $(u^*,v^*)$ is p-$\varepsilon$-minimal for \eqref{eq:abstract-problem-reduced} if $([u^*,v^*],v^*)$ is p-$\varepsilon$-optimal for \eqref{eq:abstract-problem-split-v-red}. Consistently, we say that a point $(y^*,u^*,v^*)$ is p-$\varepsilon$-minimal
for the original problem \eqref{eq:abstract-problem} if $(u^*,v^*)$ is a p-$\varepsilon$-minimum of \eqref{eq:abstract-problem-reduced} 
and $y^*$ is a solution of the state equation \eqref{eq:StateEq} with $u=u^*$ and $v=v^*$. 
We note that p-$\varepsilon$-minima are not necessarily global $\varepsilon$-minima. But any global minimum of \eqref{eq:abstract-problem}
is p-$\varepsilon$-minimal with $\varepsilon=0$. For brevity, we call p-$\varepsilon$-minima with $\varepsilon=0$ just p-minima.

The above discussion motivates to compute p-minima of good quality. To this end, we enforce the coupling of $v$ 
and $\tilde v$ in \eqref{eq:vsplit-split-v} weakly with a suitable penalty term. The penalty parameter can then eventually 
be used to avoid getting stuck in p-$\varepsilon$-minima with too high objective. This idea was introduced recently in \cite{geissler.ea:fpump-adm} 
for classical ADM in the context of feasibility pumps for MINLPs. Suitably adapted to our setting here, we are going to show an exactness result for the penalty 
problem. 

We may consider the optimal value function 
of the reduced problem \eqref{eq:abstract-problem-reduced} partially with respect to $u$
\begin{equation}
  \label{eq:abstract-problem-reduced-partial}
 \eta(v):=\inf_{u \in \Sigma \subset \Ut} \Psi(u,v)
\end{equation}
as a function $\eta\colon \Vt \to \RR \cup \{-\infty,+\infty\}$. We will impose the following technical assumption on $\eta$ using the 1-norm $\lOneNorm{v}=\sum_{i=1}^M |v_i|$ on $\{0,1\}^M$.

\begin{assumption}\label{ass:cq-implicit} Given an optimal solution $(y^*,u^*,v^*)$ of \eqref{eq:abstract-problem}, or equivalently an optimal solution $(u^*,v^*)$ of problem \eqref{eq:abstract-problem-reduced}
the value function $\eta$ defined in \eqref{eq:abstract-problem-reduced-partial} is locally Lipschitz continuous 
in the sense that for all $\delta>0$ there exists a constant $L$ such that
\begin{equation}\label{eq:eta-lipschitz}
 |\eta(v^*)-\eta(v)| \leq L \int_0^T \lOneNorm{v^*(t)-v(t)}\,dt
\end{equation}
for all $v \in \Vt$ with $\int_0^T \lOneNorm{v^*(t)-v(t)}\,dt \leq \delta$.
\end{assumption}

Assumption~\ref{ass:cq-implicit} is typically satisfied if the optimal solution $(y^*,u^*,v^*)$ satisfies a constraint qualification. For instance for mixed-integer linear quadratic optimal control problems the Lipschitz continuity of the optimal value function under a constraint qualification of a Slater-type is discussed in \cite{GugatHante2017}. For mixed-integer finite-dimensional problems, conditions are provided in \cite{Gugat1997} and \cite{HanteSchmidt2019}. 

Now we consider the following auxiliary problem
\begin{subequations}
  \label{eq:abstract-problem-split-v-penalty}
\begin{alignat}{1}
 \min_{y,u,v,\tilde v}\quad &\Phi(y(T))+ \rho \int_0^T\lOneNorm{v(t)-\tilde v(t)}\,dt\label{eq:cost-split-v-penalty}\\
 \st~&\dot{y}(t)=Ay(t)+f(y(t),u(t),v(t))\quad \text{on}~Y,~t \in (0,T),~y(0)=y_0\label{eq:StateEq-split-v-penalty}\\
 &u \in \Sigma \subset \Ut,
 ~v \in \Vt,
 ~\tilde v \in \Gamma \subset \Vt,
 ~y \in \Yt,
\end{alignat}
\end{subequations}
with a penalty parameter $\rho \geq 0$. With \eqref{eq:soldef} and \eqref{eq:cost-red} we can reduce \eqref{eq:abstract-problem-split-v-penalty} to
\begin{subequations}
  \label{eq:abstract-problem-split-v-penalty-red}
\begin{alignat}{1}
 \min_{u,v,\tilde v}\quad &\Psi_\rho(u,v,\tilde v):=\Psi(u,v)+ \rho \int_0^T\lOneNorm{v(t)-\tilde v(t)}\,dt\label{eq:cost-split-v-penalty-red}\\
 \st~& u \in \Sigma \subset \Ut,
 ~v \in \Vt,
 ~\tilde v \in \Gamma \subset \Vt.
\end{alignat}
\end{subequations}


The following result shows the exactness of the penalty term in \eqref{eq:abstract-problem-split-v-penalty} and relates global minima of \eqref{eq:abstract-problem-reduced} to p-minima of \eqref{eq:abstract-problem-split-v-penalty-red}.

\begin{theorem}\label{thm:exactness} 
  Let $(u^*,v^*)$ be a global minimum of problem
  \eqref{eq:abstract-problem-reduced} satisfying
  Assumption~\ref{ass:cq-implicit}. Then, there exists a penalty parameter
  $\bar\rho$ such that $([u^*,v^*], v^*)$ is a p-minimum of
  \eqref{eq:abstract-problem-split-v-penalty-red} for all $\rho\geq \bar{\rho}$.
\end{theorem}
\begin{proof}
Let $(u^{*}, v^{*})$ be an optimal solution of \eqref{eq:abstract-problem-reduced}. We note that by construction $u = u^{*}$, $v = v^{*}$, $\tilde{v} = v^{*}$ is a global minimum of \eqref{eq:abstract-problem-split-v-red}.

We have to show that $([u^*,v^*],\tilde v^*)$ satisfies
\begin{subequations}\label{eq:partialmindef}
\begin{alignat}{1}
\Psi_\rho(u^*,v^*,\tilde v) &\geq \Psi_\rho(u^*,v^*,\tilde v^*)-\varepsilon \quad \text{for all}~(u^*,v^*,\tilde v)~\text{feasible for \eqref{eq:abstract-problem-split-v-penalty-red}}\label{eq:partialmindefA}\\
\Psi_\rho(u,v,\tilde v^*) & \geq \Psi_\rho(u^*,v^*,\tilde v^*)-\varepsilon \quad \text{for all}~(u,v,\tilde v^*)~\text{feasible for \eqref{eq:abstract-problem-split-v-penalty-red}}\label{eq:partialmindefB}
\end{alignat}
\end{subequations}
with $\varepsilon=0$.

It can be seen that condition~\eqref{eq:partialmindefA} holds with $\varepsilon=0$ for all $\rho \geq 0$. This follows directly from the definition of $\Psi_{\rho}$.

To show condition~\eqref{eq:partialmindefB} with $\varepsilon=0$, we assume we are given $(u, v)$ and consider the triple $(u, v, {v}^{*})$.
Without loss of generality, we may assume that $u$ is chosen optimally in the sense of \eqref{eq:abstract-problem-reduced-partial}.
By the definition of $\Psi_{\rho}$, condition \eqref{eq:partialmindefB} with $\varepsilon=0$ is equivalent to
\begin{equation}
  \label{eq:exactness-partialmindefB-1}
\Psi(u,v) + \rho \int_0^T\lOneNorm{v - {v}^{*}}\,dt \geq \Psi(u^*,v^*).
\end{equation}

We directly observe that if $\Psi(u,v) \geq \Psi(u^*,v^*)$  holds, then claim~\eqref{eq:exactness-partialmindefB-1} is
true for all $\rho \geq 0$.
This is, for instance, the case if
$v = {v}^{*}$ holds, i.e. if $(u,v,{v}^{*})$ is feasible for
\eqref{eq:abstract-problem-split-v-red}, because $(u^{*}, v^{*})$ is a global minimum of \eqref{eq:abstract-problem-split-v-red}.
Hence, we only need to consider the case
in which $\Psi(u,v) \leq \Psi(u^*,v^*)$ and $v \neq {v}^{*}$ hold.

As we have chosen $u$ optimally and because $u^{*}$ is also an optimal choice for $v^{*}$ in the sense
of \eqref{eq:abstract-problem-reduced-partial}, we can rewrite~\eqref{eq:exactness-partialmindefB-1} further to obtain
\begin{equation}
  \label{eq:exactness-partialmindefB-2}
\rho \int_0^T\lOneNorm{v - {v}^{*}}\,dt \geq \eta(v^{*}) - \eta(v).
\end{equation}
Now, we may use Assumption~\ref{ass:cq-implicit} to obtain
\begin{equation*}
  \eta(v^{*}) - \eta(v)
  = \abs{\eta(v^{*}) - \eta(v)}
  \leq L  \int_0^T\lOneNorm{v - {v}^{*}}\,dt.
\end{equation*}
This shows that condition~\eqref{eq:exactness-partialmindefB-2} is fulfilled
for all $\rho \geq L$.
Hence, we have shown that condition~\eqref{eq:partialmindefB} holds with $\varepsilon=0$ if we set $\bar{\rho} = L$.
\end{proof}

The essential idea of the proposed method now is to solve \eqref{eq:abstract-problem-split-v-penalty-red} using the method
discussed at the beginning of Section~\ref{sec:epsADM}.
So, in each iteration of the outer loop (index $k$) the penalty parameter $\rho$ is increased.
In the inner loop (index $l$),
we apply an alternating direction method to \eqref{eq:abstract-problem-split-v-penalty-red} with this parameter $\rho$
until we find a partial $\varepsilon$-optimum.
For this, we need to be able to solve two subproblems to accuracy $\varepsilon$: \eqref{eq:abstract-problem-split-v-penalty-red} with
fixed $\tilde v$ and \eqref{eq:abstract-problem-split-v-penalty-red} with fixed $(u,v)$.
For fixed $\tilde{v}$ \eqref{eq:abstract-problem-split-v-penalty-red} reduces to an optimal control problem and
for fixed $(u,v)$ \eqref{eq:abstract-problem-split-v-penalty-red} reduces to an mixed-integer linear problem (assuming
the constraints describing $\Gamma$ are linear).
Both of these problem types can be solved to $\varepsilon$ accuracy with standard techniques.

The algorithm is summarized in Algorithm~\ref{alg:penalty-eps-adm}.

\begin{algorithm}
\begin{algorithmic}
\STATE Choose $(u^{(0,*)},v^{(0,*)},\tilde v^{(0,*)}) \in \Sigma \times \Vt \times \Gamma$ and $\rho^{(1)}=0$
\FOR{$k=1,2,3,\ldots$}
  \STATE Set $(u^{(k,0)},v^{(k,0)},\tilde v^{(k,0)})=(u^{(k-1,*)},v^{(k-1,*)},\tilde v^{(k-1,*)})$
  \FOR{$l=0,1,2,\ldots$}
  \STATE For $\tilde v^{(k,l)}$ fixed, find $\frac{\varepsilon}{2}$-optimal solution $(u^{(k,l+1)},v^{(k,l+1)}) \in \Sigma \times \Vt$ of \eqref{eq:abstract-problem-split-v-penalty-red}
    \IF {$\Psi((u^{(k,l+1)},v^{(k,l+1)},\tilde v^{(k,l)}) \geq \Psi((u^{(k,l)},v^{(k,l)},\tilde v^{(k,l)}))-\frac{\varepsilon}{2}$}
      \STATE Set $(u^{(k,*)},v^{(k,*)},\tilde v^{(k,*)})=(u^{(k,l+1)},v^{(k,l+1)},\tilde v^{(k,l)})$
      \STATE \textbf{break}
    \ENDIF   
    \STATE For $(u^{(k,l+1)},v^{(k,l+1)})$ fixed, find $\frac{\varepsilon}{2}$-optimal solution $\tilde v^{(k,l+1)} \in \Gamma$ of \eqref{eq:abstract-problem-split-v-penalty-red}
    \IF {$\Psi((u^{(k,l+1)},v^{(k,l+1)},\tilde v^{(k,l+1)}) \geq \Psi((u^{(k,l+1)},v^{(k,l+1)},\tilde v^{(k,l)}))-\frac{\varepsilon}{2}$}
	\STATE Set $(u^{(k,*)},v^{(k,*)},\tilde v^{(k,*)})=(u^{(k,l)},v^{(k,l)},\tilde v^{(k,l)})$
        \STATE \textbf{break}
    \ENDIF
    \ENDFOR
    \STATE Choose $\rho^{(k+1)}>\rho^{(k)}$
\ENDFOR
\end{algorithmic}
\caption{Penalty-$\varepsilon$-ADM-Method (ADM-SUR)}
\label{alg:penalty-eps-adm}
\end{algorithm}

Concerning the convergence of Algorithm~\ref{alg:penalty-eps-adm}, we can now make the following statements.

\begin{theorem} \label{thm:convergence} Let $\rho^k \upto \infty$ and let $(u^k,v^k,\tilde v^k)_k$ be a sequence generated by
Algorithm~\ref{alg:penalty-eps-adm} with $(u^k,v^k,\tilde v^k) \to (u^*,v^*,\tilde v^*)$. Then $(u^*,v^*,\tilde v^*)$ is
a p-minimum of the feasibility measure $\chi(v,\tilde v)=\int_0^T |v(t)-\tilde v(t)|_{l_1}\,dt$.
\end{theorem}
\begin{proof}
Let $(u,v,\tilde v^k)$ be feasible for \eqref{eq:abstract-problem-split-v-penalty-red}. Then 
\begin{equation}\label{eq:forconvproof}
 \Psi(u,v)+ \rho^k \int_0^T\lOneNorm{v(t)-\tilde v^k(t)}\,dt \geq \Psi(u^k,v^k)+ \rho^k \int_0^T\lOneNorm{v^k(t)-\tilde v^k(t)}\,dt - \epsilon.
\end{equation}
Let $\bar\rho$ be a cluster point of the sequence $\left(\frac{\rho^k}{|\rho^k|}\right)_k$ and $(\rho^l)_l$ be a subsequence for which 
$\left(\frac{\rho^k}{|\rho^k|}\right)_k$ converges to $\bar\rho$. Then, dividing the inequality \eqref{eq:forconvproof} by $|\rho^l|$ yields
\begin{equation*}
 \frac{\Psi(u,v)}{|\rho_l|} + \frac{\rho^k}{|\rho_l|} \int_0^T\lOneNorm{v(t)-\tilde v^k(t)}\,dt \geq \frac{\Psi(u^k,v^k)}{|\rho_l|} + \frac{\rho^k}{|\rho_l|} \int_0^T\lOneNorm{v^k(t)-\tilde v^k(t)}\,dt - \frac{\epsilon}{|\rho_l|}.
\end{equation*}
Taking the limit $l \to \infty$ yields
\begin{equation*}
 \bar\rho \int_0^T \lOneNorm{v-\tilde v^*}\,dt \geq \bar\rho \int_0^T\lOneNorm{v^*-\tilde v^*}\,dt.
\end{equation*}
An analog inequality holds for any feasible $(u^k,v^k,\tilde v)$.
\end{proof}

\begin{corollary}~\label{cor:main} Let $\rho^k \upto \infty$ and let $(u^k,v^k,\tilde v^k)_k$ be a sequence generated by
Algorithm~\ref{alg:penalty-eps-adm} with $(u^k,v^k,\tilde v^k) \to (u^*,v^*,\tilde v^*)$ and let $(u^*,v^*,\tilde v^*)$ be
feasible for \eqref{eq:abstract-problem-split-v-red}. Then $(u^*,v^*)$ is p-minimal for \eqref{eq:abstract-problem-reduced}.
\end{corollary}
\begin{proof}
This follows from Theorem~\ref{thm:convergence} and using that feasibility of $(u^*,v^*,\tilde v^*)$ for \eqref{eq:abstract-problem-split-v-red} 
implies $v^*=\tilde v^*$.
\end{proof}

Note that in the inner loop of Algorithm~\ref{alg:penalty-eps-adm}, we compute p-$\varepsilon$-minima, but that Corollary~\ref{cor:main} says
that a feasible limit of a converging sequence generated by ADM-SUR is a p-$\varepsilon$-minimum with $\varepsilon=0$. Moreover, note that
the two subproblems for \eqref{eq:abstract-problem-split-v-penalty-red} in Algorithm~\ref{alg:penalty-eps-adm} can be solved efficiently.
Finally, we note that $\rho^k \upto \infty$ is needed in Theorem~\ref{thm:convergence} and Corollary~\ref{cor:main}, 
because Theorem~\ref{thm:exactness} guarantees exactness of the penalty only in a global minimum. It must be observed that 
even in the finite-dimensional case, there are only slightly stronger results known (see \cite{geissler.ea:fpump-adm}, Theorems~8 and~11).
It is instructive to note that in the finite-dimensional case the assumption of convexity immediately yields convergence
to global optima and the assumption of differentiabilty to convergence to local optima. To us, this indicates that the mixed-integer part of the problem makes it difficult to prove anything about convergence to local (or even global) optima for these types of methods.

For any fixed $\tilde{v}$ the problem \eqref{eq:abstract-problem-split-v-penalty-red} is equivalent to the problem
\begin{subequations}
  \label{eq:POC1}
\begin{alignat}{2}
 \min_{y,z,u,w}\quad &\Phi(y(T)) + z(T)\label{eq:cost3}\\
 &\dot{y}(t)=Ay(t)+\sum_{i=1}^{\tilde M} w_i(t) f(y(t),u(t),r^i),\quad &&~t \in (0,T),~y(0)=y_0\label{eq:StateEq3}\\
 &\dot{z}(t)=\sum_{i=1}^{\tilde{M}} w_i(t) \rho |r^i-\tilde v|_{l_1},\quad &&t \in (0,T),~z(0)=0\label{eq:NewStateEq3}\\
 &u \in \Sigma \subset \Ut, \label{eq:CCC3}\\
 &\sum_{i=1}^{\tilde M} w_i(t)=1,~t\in(0,T)~\text{a.e.} \label{eq:SOS1}\\
 &y \in \Yt,~z \in \Zt,~w \in \Wt\label{eq:spaces3},
\end{alignat}
\end{subequations}
with $r_i$, $i=1,\ldots,\tilde{M}=2^M$, enumerating the configurations $\{0,1\}^M$, $\Zt=L^1(0,T)$ and $\Wt=L^\infty(0,T;\{0,1\}^{\tilde M})$. Letting
$(\bar y,\bar z,\bar u, \bar w)$ be a solution of \eqref{eq:POC1} with the relaxation $w \in L^\infty(0,T;[0,1]^{\tilde M})$ and $w^n\in \Wt$ be a sequence
generated by the sum-up rounding algorithm of \cite{Sager2006,SagerBockDiehl2012} and $y_n$, $z_n$ be the corresponding solutions of \eqref{eq:StateEq3} and
\eqref{eq:NewStateEq3} with $w=w^n$, then under Assumption~\ref{ass:main}
\begin{equation}\label{eq:POCconvergence}
 \|\bar y-y_n\|_{C([0,T];Y)}+\|\bar z-z_n\|_{C([0,T])} \to 0,\quad\text{for}~n\to\infty,
\end{equation}
see \cite{MannsKirches2019} for details. Under additional assumptions on $A$ and $f$, even error estimates are available \cite{SagerBockDiehl2012,HanteSager2013,Hante2017}.
In particular, \eqref{eq:POCconvergence} shows that this solution approach yields an $\frac{\epsilon}{2}$-optimal solution for a sufficiently fine control grid.
We refer to this solution approach for subproblem \eqref{eq:POC1} as the (POC)-step.

Further, for any fixed $(u,v)$ the problem \eqref{eq:abstract-problem-split-v-penalty-red} reduces to
\begin{alignat}{2}
  \label{eq:MIP}
 \min_{\tilde{v} \in \Gamma}\quad &\int_0^T |v-\tilde{v}|_{l_1}\,dt.
\end{alignat}
Here, standard quadrature rules and mixed-integer linear programming techniques can be used to compute an $\frac{\varepsilon}2$-optimal solution again for a sufficiently fine
control grid. We refer to this solution approach for the subproblem \eqref{eq:MIP} as the (MIP) step.

The sum-up rounding algorithm in the (POC)-step can be interpreted as the solution of the following combinatorial integral approximation problem (CIAP)

\begin{alignat}{2}
\min_{w^n} \max_{t \in (0,T)} \quad & \left\lVert \int_0^t w(s) - w^n(s) ds \right\rVert_{\infty} \label{eq:costSURCIAP},
\end{alignat}
for a piecewise constant function $w^n$ on a fixed grid, see \cite{SagerJungKirches2011}. We therefore refer to the penalty-$\varepsilon$ method in Algorithm~\ref{alg:penalty-eps-adm} as ADM-SUR. 
An interesting variant of this Algorithm is to skip SUR in the POC-step, \ie doing the step in the relaxed direction of $w$ and using the MIP-step to recover integer feasibility. We refer
to this variant as ADM (without SUR). For comparison, we also consider the heuristic to apply POC to the original problem formulation without combinatorial constraints and to recover
a feasible solution via the following mixed-integer problem
\begin{subequations}
  \label{eq:MIPheuristic}
\begin{alignat}{2}
\min_{w^n} \max_{t \in (0,T)} \quad & \left\lVert \int_0^t w(s) - w^n(s) ds \right\rVert_{\infty}\label{eq:CIAPheuristic}\\
 &w^n \in \Gamma \label{eq:CCC5},
\end{alignat}
\end{subequations}
again on a fixed grid for $w^n$, see \cite{SagerJungKirches2011,HanteSager2013}. We refer to this approach as CIAP. Note that in contrast to \eqref{eq:MIP}, the cost function in \eqref{eq:MIPheuristic}
is not a norm on $\Wt$. A convergence result for this subproblem in an ADM framework such as for the ADM-SUR algorithm is therefore an open problem.


\section{Numerical study}\label{sec:numstudy}

We test the proposed methods on two benchmark examples from the
\href{https://mintoc.de}{mintOC.de} library \cite{sager2012benchmark}, which we
augment by minimum dwell-time constraints in order to prohibit infinitely many
switching events.
To model the dwell time condition we used the basic MIP constraints.
These do not, however, form a complete description of the so-called min up/down polyhedron.
One can use either use a complete description in the original variable space~\cite{lee2004min}, where
additional constraints are separated with cutting planes or use an extended formulation~\cite{rajan2005minimum} instead.
As the main focus of this article is the solution quality, we stick to the basic
formulation above.

Our computations are based on CasADi~\cite{Andersson2019} for the model
equations and their derivatives and the solvers Ipopt~\cite{WaechterBiegler2006}
for nonlinear programming problems and Gurobi~\cite{gurobi} for quadratic and
linear mixed-integer programs.

The ADM method was used with $\rho = 10^{-3}, 10^{-2}, \dotsc, 10^6$ and $\varepsilon
= 10^{-3}$.
We terminated the method when the value of the penalty term dropped below a tolerance of $10^{-4}$, because increasing $\rho$ beyond that point will not change the iterates much on a fixed discretization.
At the end of this section, we study the dependency of the penalty parameter adaptation for smaller choices of the multiplicative increment.
Our study confirms the experience from the finite-dimensional case (the version described in \cite{geissler.ea:fpump-adm}) that the effect of the penalty adaptation strategy does not change the qualitiative behavior of the method. To our point of view, several different strategies can be used to identify $p$-minima with low objectives for example within a global search strategy like branch-and-bound. An interesting direction for further research seems to be the use of weighted $1$-norm-penalties with adaption strategies for the weights as used in the finite-dimensional case (see, for instance, \cite{geissler.ea:powernova,geissler.ea:compressor-adm}). This is not straight-forward in the infinite-dimensional case, because that would make the strategy discretization-dependent. 

\subsection{Fuller's problem}

\begin{figure}[tb]
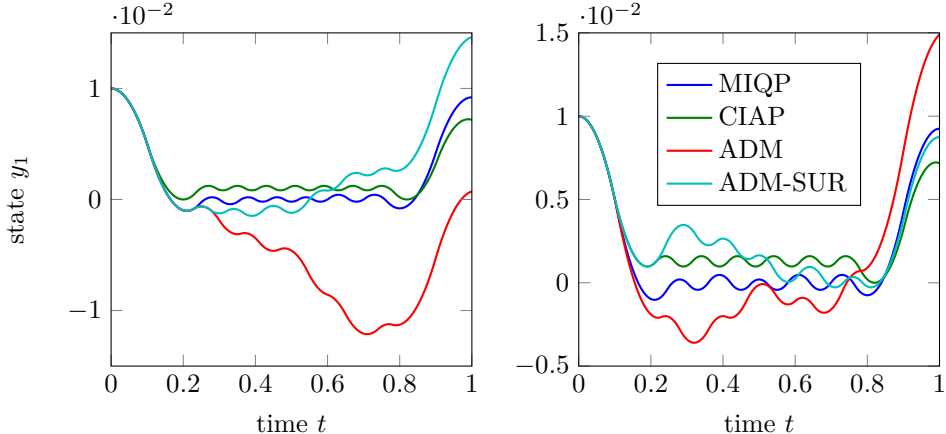

  \centering
  \input{fuller_states_tau=0.04.tex}
  \input{fuller_states_tau=0.05.tex}
  \caption{
    Results for Fuller's problem with minimum dwell-times $\tau_\mathrm{min} =
    0.04$ (left) and $\tau_\mathrm{min} = 0.05$ (right).
    The heuristic results based on ADM and CIAP can in a qualitative sense get close to the global solution computed with an MIQP solver.}
  \label{fig:fuller_states}
\end{figure}

For our numerical study, we consider a variant of Fuller's problem augmented with minimum dwell-time constraints
\begin{subequations}
  \label{eqn:fuller}
  \begin{alignat}{2}
    \min_{y,v} & \int_{0}^{1} y_1(t)^2 dt + \left(y_1(1) - \tfrac{1}{100}\right)^2 + y_2(1)^2\label{eq:fuller-objective}\\
    \text{s.t.~} & \dot{y}_1(t) = y_2(t),\quad t \in (0,1)\label{eq:fuller-ode-1}\\
    & \dot{y}_2(t) = 1 - 2 v(t),\quad t \in (0,1)\label{eq:fuller-ode-2}\\
    & y(0) = \left(\tfrac{1}{100}, 0\right)^\top,\\
    & v(t) \in \{0, 1\},\quad t \in (0,1)\\
    & v(t+\tau_{\min})-v(t)+|v|_{(t,t+\tau_{\min})} \leq 2,\quad t \in (0,1-\tau_{\min}).
  \end{alignat}
\end{subequations}
The problem is notoriously difficult, because the solution of the problem without dwell-time constraints 
(i.e., for $\tau_{\min}=0$) exhibits chattering \cite{ZelikinBorisov1994}.

We compare our proposed ADM-based method (with and without SUR) with a direct
global Mixed-Integer Quadratic Programming (MIQP) method and CIAP.
We discretize Equations~\eqref{eq:fuller-ode-1} and \eqref{eq:fuller-ode-2} using a Gauss--Legendre collocation of degree 4 on an equidistant partition of $[0,1]$ with 200 collocation intervals. The same collocation nodes are also used for approximating the integral term in~\eqref{eq:fuller-objective} via Gauss--Legendre quadrature. The control discretization is piecewise constant, with jumps allowed only at the boundary of the collocation intervals but not at the collocation nodes. Because the objective~\eqref{eq:fuller-objective} is quadratic in $y$ and the constraints~\eqref{eq:fuller-ode-1} and~\eqref{eq:fuller-ode-2} are linear in $y$, the same holds for their discretized counterparts. Hence, we obtain a discretized
MIQP, which we solve, where possible, to global optimality using Gurobi.
The resulting objective values and
corresponding single CPU runtimes are depicted in
Tables~\ref{tab:objective_values} and~\ref{tab:runtimes}. It appears that the
ADM-based methods have some advantage both in quality and runtime over CIAP for the
instances with larger $\tau_\mathrm{min}$, which are harder for POC-based heuristics (but appear to be simpler for the MIQP approach).
Exemplary for two selected values dwell-times $\tau_{\min}$, the resulting state $y_1$ in problem~\ref{eqn:fuller}
is shown in Figure~\ref{fig:fuller_states}.

\begin{table}
    \caption{Comparison of the objective function values for the four approaches
  on a Gauss--Legendre collocation discretization of degree 4 on an equidistant
  grid with 200 intervals for Fuller's problem \eqref{eqn:fuller} with respect
  to varying values of the minimum dwell time $\tau_\mathrm{min}$. The best
  objective value among the heuristic approaches is highlighted in boldface.}
  \label{tab:objective_values}
  \centering
  \begin{tabular}{ccccc}
    \toprule
    $\tau_\mathrm{min}$ & MIQP & CIAP & ADM & ADM-SUR\\
    \midrule
    0.01 & {  0.014508} & \textbf{  0.014870} & {  0.016653} & {  0.498363} \\
    0.02 & {  0.014511} & \textbf{  0.130346} & {  0.493694} & {  0.432311} \\
    0.03 & {  0.014517} & \textbf{  0.116714} & {  1.182971} & {  0.467442} \\
    0.04 & {  0.014530} & \textbf{  0.120164} & {  0.234605} & {  0.148813} \\
    0.05 & {  0.014558} & {  0.120706} & {  0.450784} & \textbf{  0.016739} \\
    0.06 & {  0.014649} & {  0.116457} & {  0.831939} & \textbf{  0.015566} \\
    0.07 & {  0.014666} & {  0.954087} & { 10.489119} & \textbf{  0.540208} \\
    0.08 & {  0.015027} & {  0.426618} & { 18.972511} & \textbf{  0.039570} \\
    0.09 & {  0.015027} & {  0.137513} & {  0.157761} & \textbf{  0.017543} \\
    0.10 & {  0.015173} & {  0.209153} & \textbf{  0.149268} & {  1.090531} \\
    \bottomrule
  \end{tabular}
\end{table}

\begin{table}
    \caption{Single CPU runtimes in seconds on an Intel(R) Core(TM) i7-5820K CPU @
    3.30GHz for the corresponding results in Table~\ref{tab:objective_values}.
    The remaining MIP gap achieved by Gurobi at a timeout of 1 hour is given in
    parantheses. Even though the codes for the heuristics CIAP, ADM, and
    ADM-SUR have not been heavily optimized, the runtimes are much smaller
  than for the MIQP solver.}
  \label{tab:runtimes}
  \centering
  \begin{tabular}{ccccc}
    \toprule
    $\tau_\mathrm{min}$ & MIQP & CIAP & ADM & ADM-SUR\\
    \midrule
    0.01 & 3600.00 (0.021\% MIP gap) & 5.49 & \textbf{1.47} & 2.53 \\
    0.02 & 3600.00 (0.031\% MIP gap) & 8.33 & \textbf{2.22} & 2.64 \\
    0.03 & 2720.16  & 15.70 & \textbf{2.35} & 3.55 \\
    0.04 & 481.51  & 10.92 & \textbf{2.85} & 3.91 \\
    0.05 & 174.60  & 12.32 & \textbf{2.71} & 4.08 \\
    0.06 & 149.19  & 11.56 & 2.96 & \textbf{2.78} \\
    0.07 & 70.82  & 13.33 & \textbf{2.74} & 3.73 \\
    0.08 & 94.71  & 9.80 & \textbf{2.90} & 3.31 \\
    0.09 & 45.43  & 14.84 & \textbf{2.36} & 5.44 \\
    0.10 & 45.03  & 11.34 & \textbf{2.59} & 6.24 \\
    \bottomrule
  \end{tabular}
\end{table}

\subsection{Network of transmission lines}

\begin{figure}[tb]
  \centering
  \begin{tikzpicture}
    \tikzset{v/.style={shape=circle,fill=black,minimum width=3pt,inner sep=0pt}}
    \tikzset{e/.style={thick,->,>=latex'}}
    \node[v,label=left:Producer $u_1(t)$] (generator1) at (0,2.5) {};
    \node[v,label=left:Producer $u_2(t)$] (generator2) at (0,0.5) {};
    \node[v,label=right:Consumer 1] (consumer1) at (5,0) {};
    \node[v,label=right:Consumer 2] (consumer2) at (5,0.5) {};
    \node[v,label=right:Consumer 3] (consumer3) at (5,2.5) {};
    \node[v,label=right:Consumer 4] (consumer4) at (5,1.8) {};
    \node[v,label=right:Consumer 5] (consumer5) at (5,1.2) {};
    \node[v] (top1) at (1,2.5) {};
    \node[v] (sub1) at (1,1.7) {};
    \node[v] (sub2) at (2,1.5) {};
    \node[v] (sub3) at (1,1.3) {};
    \node[v] (mid1) at (4,1.5) {};
    \node[v] (bottom1) at (1,0.5) {};
    \node[v] (bottom2) at (4,0.5) {};

    \draw[e] (generator1) -- (top1);
    \draw[e] (generator2) -- (bottom1);

    \draw[e] (top1) -- (consumer3);
    \draw[e,dashed] (top1) -- node[right] {$v_1(t)$} (sub1);
    \draw[e,dashed] (bottom1) -- node[right] {$v_1(t)$} (sub3);
    \draw[e] (bottom1) -- (consumer1);

    \draw[e] (sub1) -- (sub2);
    \draw[e] (sub3) -- (sub2);
    \draw[e] (bottom1) -- (bottom2);

    \draw[e,dashed] (sub2) -- node[above] {$v_1(t)$} (mid1);

    \draw[e,dotted] (bottom2) -- node[left] {$v_2(t)$} (mid1);

    \draw[e] (mid1) -- (consumer4);
    \draw[e] (mid1) -- (consumer5);
    \draw[e] (bottom2) -- (consumer2);
  \end{tikzpicture}
  \caption{Network topology for the subgrid scenario of the transmission lines
    example. The objective is to continuously control the power generation at
    the producers via $u_1(t)$ and $u_2(t)$ and to switch on or off the dashed
    connections (via $v_1(t)$) and the dotted connection (via $v_2(t)$) in order
    to minimize the quadratic deviation of the power supply from the power
    demand at the five consumer nodes.}
  \label{fig:subgrid}
\end{figure}
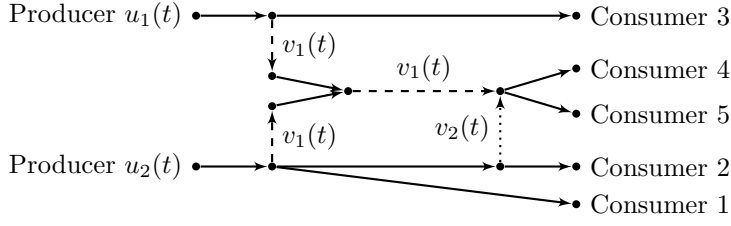

\begin{table}[!htb]
    \caption{Scaled objective values for different approaches to the transmission
  lines example. POC and SUR violate the minimum dwell-time constraints. The
  ADM-based heuristics deliver the best objective values on the subgrid
  scenario. For the extended tree scenario, cf.~\cite{GoettlichPotschkaTeuber2019}, all heuristics perform equally
  well.}
  \label{tab:translines_obj_vals}
  \centering
  \begin{tabular}{ccccccc}
    \toprule
    Scenario & POC & SUR & CIAP & ADM & ADM-SUR\\
    \midrule
    Subgrid & 1.538 & 5.483 & 5.256 & 3.719 & 3.370 \\
    Extended tree & 2.775 & 3.103 & 3.096 & 3.081 & 3.078 \\
    \bottomrule
  \end{tabular}
\end{table}

This problem was described in \cite{GoettlichPotschkaTeuber2019}.
The telegraph equations are based on a $2\times2$ hyperbolic system of partial differential equations
and describe the voltage and current on electrical transmission lines in time $t\in[0,T]$ and space $x\in[0,l]$.
The state variable $\boldsymbol{\xi}(x,t)=(\xi^+(x,t),\xi^-(x,t))$ represents right or left-traveling components on
{\em each} line of the network and is governed by 
\begin{equation}
\label{arc}
\partial_t \boldsymbol{\xi} + \boldsymbol{\Lambda} \partial_x \boldsymbol{\xi} + \mathbf{B} \boldsymbol{\xi} = 0,
\end{equation}
where $\boldsymbol{\Lambda}$ is a diagonal matrix including the speed of propagation in each direction and $\mathbf{B}$ denotes 
a symmetric matrix with non-negative entries.
The dynamics on the lines are coupled at nodes via the boundary condition 
\begin{equation}
\label{coupling}
\begin{pmatrix}
\boldsymbol{\Lambda}^+ & 0\\
0 & \mathbf{D}^-(\boldsymbol{v(t)})
\end{pmatrix} \boldsymbol{\xi}(0,t) = 
\begin{pmatrix}
\mathbf{D}^+(\boldsymbol{v(t)}) & 0\\
0 & \boldsymbol{\Lambda}^-
\end{pmatrix} \boldsymbol{\xi}(l,t) + 
\begin{pmatrix}
\boldsymbol{\Lambda}^+ & 0\\
0 & 0
\end{pmatrix} \boldsymbol{u}(t).
\end{equation}
The distribution matrices $\mathbf{D}^{\pm}(\boldsymbol{v})$ depend on binary-valued controls
$\boldsymbol{v}(t)\in\{0,1\}$, which are used to switch
off specified connections in the network while
the continuous-valued controls $\boldsymbol{u}(t)$ denote
the power generation at the producer nodes in the network, cf. Figure \ref{fig:subgrid}. 
The goal is to minimize the quadratic deviation of the accumulated power delivery $C_s(t,\boldsymbol{\xi})=\sum_{r \in \delta_S} \xi^+_r(l_r,t)$ 
(with $\delta_S$ being the set of all lines adjacent to node $s$) from the demand $Q_s(t)$ at the
consumer nodes $V_S$, i.e.,
\begin{equation}
\label{optprob}
\begin{cases}
		\quad \min\limits_{\boldsymbol{v,u}} \frac{1}{2}\sum\limits_{s \in V_S} \int_0^{T} \left( Q_s(t) - C_s(t,\boldsymbol{\xi}) \right)^2 dt\\
\\
		\quad \text{s.t. }	 (\ref{arc}) \text{ and }  (\ref{coupling}).
		\end{cases}
\end{equation}
This problem can be written in abstract form as
\begin{equation}\label{eq:telegraph_abstract}
 \dot{y}=Ay+B(v)u,
\end{equation}
with $A$ and $B(v_i)$, $i=1,\ldots,M$ being unbounded linear operators on Hilbert spaces using abstract semigroup theory \cite{Bartecki2015}.
Though \eqref{eq:telegraph_abstract} is not of the form \eqref{eq:abstract-problem}, the solution is still
given by the variation of constants formula \eqref{eq:soldef} with $f(y,u,v)=B(v)u$, see \eg \cite{CurtainZwart1995}.

\begin{figure}[tb]
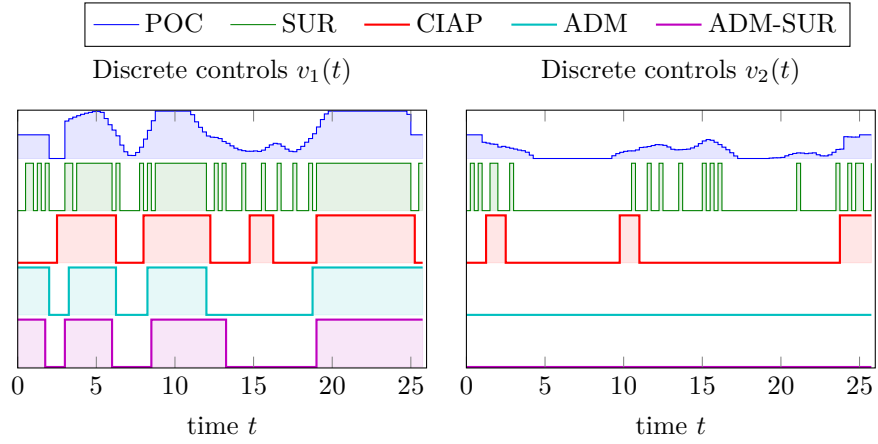

    \centering
    \begin{tikzpicture}
  \definecolor{color0}{rgb}{0,0.75,0.75}
  \definecolor{color1}{rgb}{0.75,0,0.75}
  \begin{axis}[
    width=\textwidth,
    height=2cm,
    hide axis,
    legend cell align={left},
    legend entries={{POC},{SUR},{CIAP},{ADM}, {ADM-SUR}},
    legend columns=6,
    legend style={/tikz/every even column/.append style={column sep=8pt}},
    xmin=0, xmax=26,
    ymin=-0.02, ymax=5.42
  ]
    \addlegendimage{no markers, blue}
    \addlegendimage{no markers, green!50.0!black}
    \addlegendimage{thick, no markers, red}
    \addlegendimage{thick, no markers, color0}
    \addlegendimage{thick, no markers, color1}
  \end{axis}
\end{tikzpicture}
    \input{translines_discrete_control_1}
    \input{translines_discrete_control_2}
    \caption{The resulting binary-valued controls in the transmission lines subgrid scenario
    for different solution approaches: The partially outer convexified
    relaxed solution (POC) delivers a lower bound, but is not binary feasible.
    The application of Sum-Up Rounding (SUR) yields binary feasible controls,
    which oscillate heavily and do not satisfy the minimum dwell-time
    constraint, however. The ADM-based heuristics result in fewer switches than
    the heuristic based on solving a Combinatorial Integral Approximation
    Problem with minimum dwell-time constraints (CIAP).}
    \label{fig:translines_discrete_controls}
\end{figure}
\begin{figure}[tb]
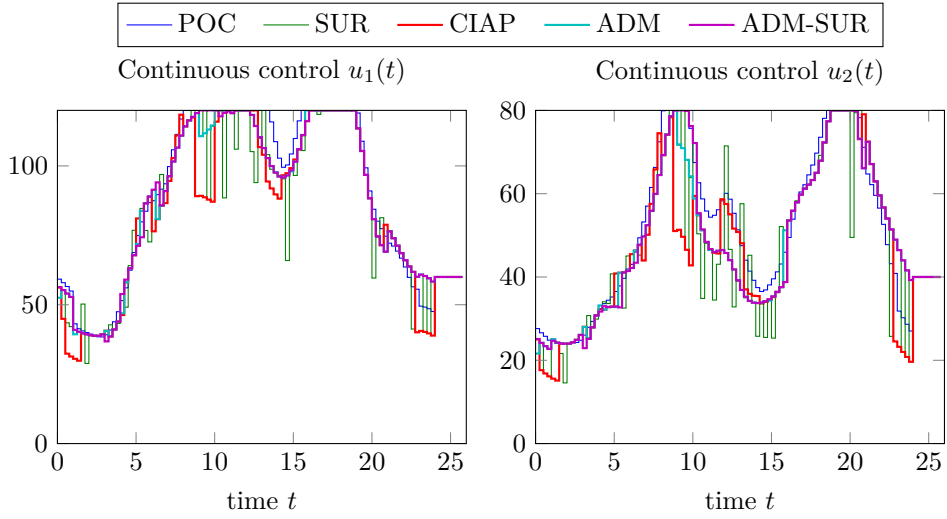

    \centering
    \begin{tikzpicture}
  \definecolor{color0}{rgb}{0,0.75,0.75}
  \definecolor{color1}{rgb}{0.75,0,0.75}
  \begin{axis}[
    width=\textwidth,
    height=2cm,
    hide axis,
    legend cell align={left},
    legend entries={{POC},{SUR},{CIAP},{ADM}, {ADM-SUR}},
    legend columns=6,
    legend style={/tikz/every even column/.append style={column sep=8pt}},
    xmin=0, xmax=26,
    ymin=-0.02, ymax=5.42
  ]
    \addlegendimage{no markers, blue}
    \addlegendimage{no markers, green!50.0!black}
    \addlegendimage{thick, no markers, red}
    \addlegendimage{thick, no markers, color0}
    \addlegendimage{thick, no markers, color1}
  \end{axis}
\end{tikzpicture}
    \input{translines_continuous_control_1}
    \input{translines_continuous_control_2}
    \caption{The continuous-valued controls oscillate on a similar scale as the
      binary-valued controls in Figure~\ref{fig:translines_discrete_controls}. The
      ADM-based results exhibit much smaller jumps.}
    \label{fig:translines_continuous_controls}
\end{figure}

For the computational experiments, we use the publicly available\footnote{See
\url{https://github.com/apotschka/poc-transmission-lines}.}
Python implementation, which uses a classical upwinding Finite Volume discretization with 4 equidistant volumes per line with forward Euler timestepping with 104 equidistant time steps as in~\cite{GoettlichPotschkaTeuber2019}.
The minimum dwell-time constraints are set to $\tau_\mathrm{min} = 1$.


The Figures~\ref{fig:translines_discrete_controls}--\ref{fig:translines_demand_and_delivery}
illustrate the results for a scenario, in which a small subgrid of the network
can be islanded, see Figure~\ref{fig:subgrid}. We observe that the binary
decisions can be partly equalized by reactions in the power generation at the
producer nodes.

\begin{figure}[tb]
    \centering
    \begin{tikzpicture}
  \definecolor{color0}{rgb}{0,0.75,0.75}
  \definecolor{color1}{rgb}{0.75,0,0.75}
  \begin{axis}[
    width=\textwidth,
    height=2cm,
    hide axis,
    legend cell align={left},
    legend entries={{POC},{SUR},{CIAP},{ADM}, {ADM-SUR}},
    legend columns=6,
    legend style={/tikz/every even column/.append style={column sep=8pt}},
    xmin=0, xmax=26,
    ymin=-0.02, ymax=5.42
  ]
    \addlegendimage{no markers, blue}
    \addlegendimage{no markers, green!50.0!black}
    \addlegendimage{thick, no markers, red}
    \addlegendimage{thick, no markers, color0}
    \addlegendimage{thick, no markers, color1}
  \end{axis}
\end{tikzpicture}
    \input{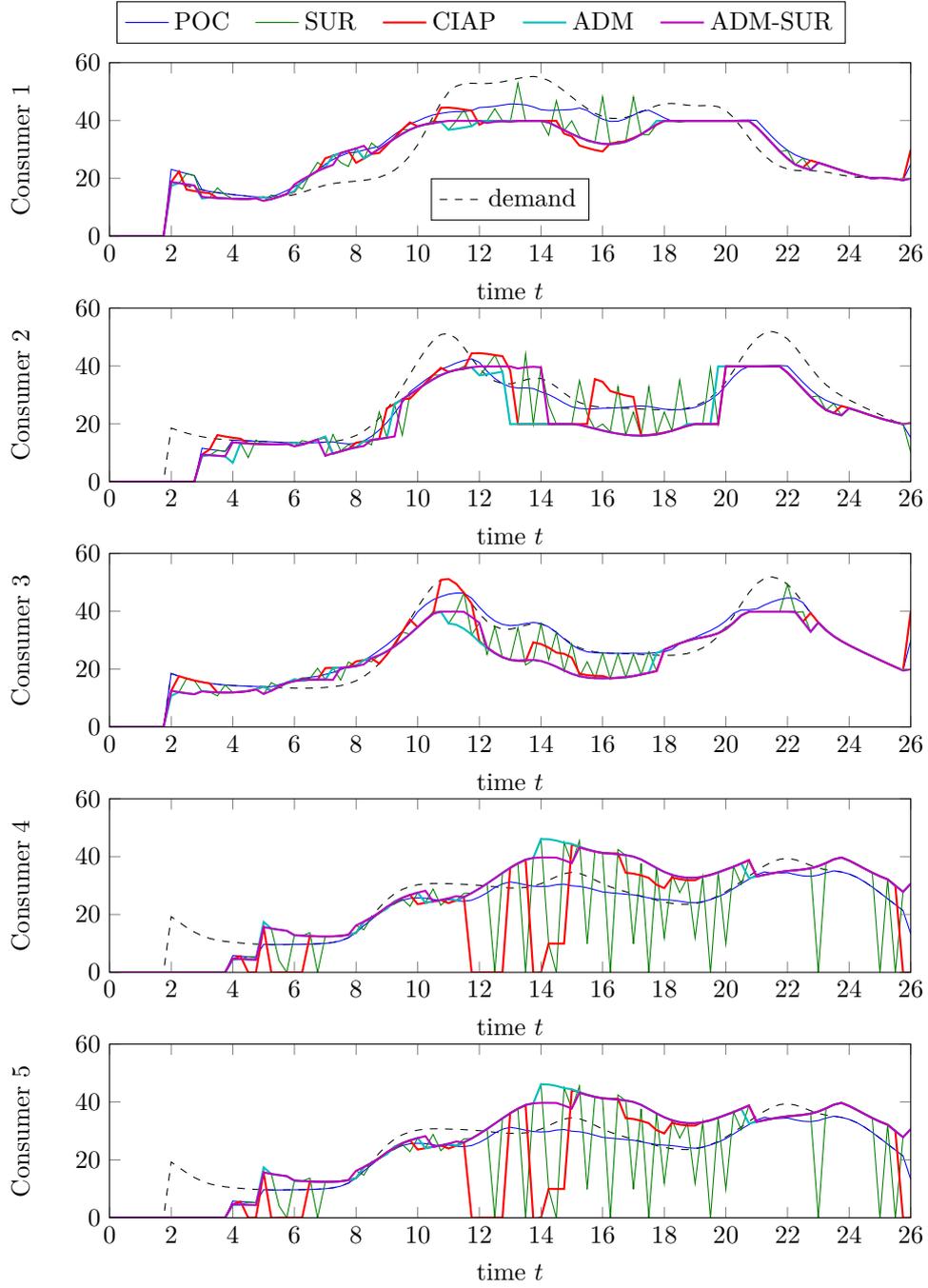}
    \caption{Resulting power supply for the controls from Figure~\ref{fig:translines_discrete_controls}
    and Figure~\ref{fig:translines_continuous_controls}. Due to the additional minimum dwell-time constraints, the deviation
    of power delivery from the demand at consumer nodes is raised in comparison
    to the lower bound given by POC. The ADM-based heuristical results are
    superior to both SUR (which does not satisfy the minimum dwell-time
    constraint) and CIAP.}
    \label{fig:translines_demand_and_delivery}
\end{figure}

Finally, we present a numerical study of the influence of the penalty parameter adaptation on the resulting objective function in Table~\ref{tab:rho_influence}. To this end, we use a coarser discretization of the subgrid szenario (2 equidistant finite volumes per transmission line, 52 equidistant time steps) and increase the penalty parameter in multiplicative steps of $\sqrt[i]{10}$ for varying $i \in \{ 1, 2, 4, 8 \},$ starting from $\rho = 10^{-3}$. The outer loop is terminated when the value of the penalty term drops below $10^{-4}$. We observe that the ADM without CIAP is largely unaffected by the penalty adaptation choice, while the ADM with CIAP shows a more pronounced dependence.

\begin{table}[tbp]
  \caption{The final value $\rho^\ast$ and the resulting objective value $\Phi^\ast$ for the ADM without CIAP are only marginally influenced by the choice of increment factor in the adaptation of the penalty parameter $\rho$. The dependence for the ADM with CIAP is more pronounced.}
  \label{tab:rho_influence}
  \centering
  \begin{tabular}{ccccc}
    \toprule
    $\rho$ incr.~factor & $\rho^\ast$ (w/o CIAP) & $\rho^\ast$ (CIAP) & $\Phi^\ast$ (w/o CIAP) & $\Phi^\ast$ (CIAP)\\
    \midrule
    $10$ & 10 & 10
    & 3.8297629676
    & 4.0293477718
    \\
    $\sqrt{10} \approx 3.16$ & 3.16 & 3.16
    & 3.8297628767
    & 3.3676874444
    \\
    $\sqrt[4]{10} \approx 1.79$ & 3.16 & 3.16
    & 3.8297628767
    & 3.5845405543
    \\
    $\sqrt[8]{10} \approx 1.33$ & 3.16 & 2.37
    & 3.8297628767
    & 3.9129098172
    \\
    \bottomrule
  \end{tabular}
\end{table}


\clearpage
\section{Conclusion}\label{sec:conclusion}
We conclude that the proposed penalty-ADM method performs notably well for our benchmark problems within the class of 
mixed-integer optimal control problems with dwell-time constraints. The quality of the computed solutions 
outperforms the other considered heuristic solutions for large dwell-times.
We think that it is worthwhile to use this heuristic inside of exact methods
to ensure that good feasible solutions are found early on in the solution process.
Moreover, the convergence theory shows that the proposed method computes
partial minima in a lifted sense. The comparison with a global solution for a full discretization shows that these
partial minima are in general not global minima. However, we note that this is not surprising because we used a local solver for
the POC-step. The proposed methods can be extended in various directions such as considerations of state constraints, mixed-integer corrector steps
from linearizations and of course more general problem classes.

\medskip
\noindent\textbf{Acknowledgements}. The second and fourth author were supported by the Deutsche Forschungsgemeinschaft (DFG) within the Sonderforschungsbereich/Transregio~154 \emph{Mathematical Modelling, Simulation and Optimization using the Example of Gas Networks}, Projects A03 and B07.
The research  of the fourth author has been performed as part of the
Energie Campus Nürnberg and is supported by funding of the Bavarian
State Government.
The third author was supported by the German Federal Ministry for Education (BMBF) and
Research under grants MOPhaPro (05M16VHA) and MOReNet (05M18VHA) while the first author 
was supported by the BMBF under grant ENets (05M18VMA).



\end{document}